\documentclass[a4paper, 12pt]{article}
\RequirePackage{amssymb,amsmath, mathtools, tikz, amscd}
\usetikzlibrary{matrix}

\usepackage[labelfont=bf]{caption}
\usepackage[margin=2.5cm, top=4cm]{geometry}
\usepackage{mathrsfs}
\usepackage{amsthm}
\usepackage{wasysym}
\usepackage{amsmath}
\usepackage{comment}

\newtheorem{Prop}{Proposition}
\newtheorem*{Prop*}{Proposition A}
\newtheorem{thm}{Theorem}
\newtheorem*{thm*}{Theorem}
\newtheorem{lemma}{Lemma}
\newtheorem*{lemma*}{Lemma A}
\newtheorem{cor}{Corollary}
\newtheorem*{cor*}{Corollary A}
\newtheorem*{Def}{Definition}
\newtheorem*{cfa}{Class Field Theory Axiom}
\begin{document}
		\title{Neukirch's Approach to the Reciprocity Map and Global Class Field theory in Positive Characteristic}
		\author{Seok Ho Yoon}
		\maketitle
		\begin{abstract}
		\emph{Neukirch developed an axiomatic and explicit approach to class field theory. This was applied to local fields and number fields but was never done for global function fields since he believed that geometric approach is more suitable. Nonetheless all the axioms can be verified and we plan to prove that global function fields satisfy the axioms Neukirch developed. }
		\end{abstract}
			\section{Introduction}
			
			Jüergen Neurkirch's approach to class field theory plays a fundamental role in explicit class field theory.  Neukirch's mechanism of class field theory provides an easy and explicit construction of the reciprocity map from several axioms. His approach is well presented in his research paper (\cite{NKFT}) and two books (\cite{NCFT} and \cite{NANT}).
			
			Recently, in view of Shinichi Mochizuki's IUT theory, a new interest to Neukirch's class field theory emerged, since Neukirch's mechanism does not use ring structures but Galois modules structure only. This is explained in detail in \cite[p.~8$-$10]{Rims}.
			
			In his original work, Neukirch applied his method to class field theory of local fields and of number fields. The case of global fields of positive characteristic, i.e. functions fields of curves over finite fields was never published. 
			
			Ivan Fesenko generalised Neukirch's method to higher class field theory (where relevant objects do not satisfy the Galois descent property) and developed class field theory of higher local fields, as well as other higher fields (\cite{HCFTI},\cite{HCFTP},\cite{HCFT0},\cite{HCFT}). He also generalised Neukirch's approach in another direction to apply to class field theory of complete discrete valuation fields with perfect residue field and more generally with imperfect residue field (\cite{PCFT},\cite{OGLR})

			This paper is dedicated to class field theory of global fields $K$ of positive characteristic $p$ using Neukirch's mechanism of class field theory.

			\newpage
		\section{Class Field Theory Axiom}
		Here we recall the main idea behind Neukirch's method on Class field theory. We shall briefly outline the abstract class field theory Neukirch constructs and how this applies to local fields. But before this, we first define and prove some results on $H^0$ and $H^{-1}$. Any omitted proof in this section can be found in \cite{NANT}.
		\subsection{$H^0$ and $H^1$}
		Let $G$ be a group and let $A$ be a $G$-module. If $G$ is finite then every $G$-module $A$ contains the norm group 
		$$N_GA=\{N_Ga:=\sum_{\sigma\in G}\sigma a|a\in G\}.$$ Since $\tau N_Ga=\sum_{\sigma\in G} \tau \sigma a =N_Ga$, we have $N_G A\subset A^G$. Then we define the norm residue group
		$$H^0(G,A)=A^G/N_G A.$$
		
		Now let $G$ be a cyclic group and $\sigma$ a generator of $G$. If $A$ is a $G$-module, then, in addition to $H^0(G,A)$ we consider the group 
		$$H^{-1}(G,A)={}_{N_G}A/I_GA$$ where ${}_{N_G}A=\{a\in A|N_Ga=0\}$ and $I_GA=\{\sigma a-a|a\in A\}.$ Those two objects are essential in class field theory and the objects $H^{i}(G,A)$ can be defined more abstractly for all $i\in \mathbb{Z}$. See $\cite{ANTC}$ for the full definition.
		
		Another concept which will come very useful is the Herbrand quotient. For $G$, a finite cyclic group and $A$ a $G$-module,we define the Herbrand quotient of $A$ to be
		$$h(G,A)=\frac{\#H^0(G,A)}{\#H^{-1}(G,A)}$$ provided that both these orders are finite. Here we state two properties of the Herbrand quotient:
		\begin{Prop}\label{herbrand}
			Let $G$ be a finite cyclic group and let $0 \rightarrow A \rightarrow B \rightarrow C \rightarrow0$ be an exact sequence of $G$-modules. Then
			$$h(G,B)=h(G,A)\cdot h(G,C),$$ in the sense that when two of these quotients are defined, so is the third and the equality holds. If $A$ is a finite $G$-module, then h$(G,A)=1$.
		\end{Prop}
			Now we introduce semilocal theory which will be useful for the transition between local and global class field theory. The rest of content in this chapter is taken from \cite[p.~183$-$185]{Lang} .
			
			We consider a finite group $G$ acting on an abelian group $A$. Assume that $A$ is direct product of subgroups, $$A=\prod_{i=1}^{s}A_i,$$ and that $G$ permutes these groups $A_i$ transitively. Let $G_1$ be the decomposition group of $A_1$, i.e. subgroup of elements which fix $A_1$. We call ($G_1,A_1)$ its local component. For each $a\in A$ we can uniquely write $a=\sum_{i=1}^s a_i$ where $a_i\in A_i$. For each $i$ we choose $\sigma_i$ such that $\sigma_iA_1=A_i$. Then we have left coset decomposition 
			$$G=\cup_{i=1}^s\sigma_iG_1.$$ We note that each $a_i\in A_i$ can be written as $\sigma_ia_i'$ for a uniquely determined element $a_i'\in A_1$.
			\begin{lemma}
				The projection $\pi:A\rightarrow A_1$ induces an isomorphism $H^0(G,A)\cong H^0(G_1,A_1)$.
			\end{lemma}
			\begin{proof}
				We first observe that $A^G$ consists of all elements of the form 
				$$\sum_{i=1}^{s}\sigma_ia_1$$ with $a_1 \in A_1^{G_1}$. It is clear that such an element is fixed under $G$. On the other hand if $a=\sum_{i=1}^{s}\sigma_ia_i'$ where $a_i'\in A_1$ for each $i$, is fixed under $G$, then for a fixed index $j$ we apply $\sigma_j^{-1}$ and see that $a_j'=\sigma_j^{-1}\sigma_ja_j'$ is the $A_1$-component of $\sigma_j^{-1}a=a$. Hence $a'_j=a_1'$ for all $j$. In particular, an element of $A^G$ is uniquely determined by its first component, and thus the projection gives an isomorphism 
				$$A^G \cong A_1^{G_1}.$$ On the other hand, for a fixed $j$ and $a_1 \in A_1$ we have
				$$N_G(\sigma_ja_1)=\sum_{\sigma\in G}\sigma a_1=\sum_{i=1}^{s}\sigma_iN_{G_1}(a_1).$$ This shows that $N_G(A)$ consists precisely of those elements of the form
				$$\sum_{i=1}^{s}\sigma_iN_{G_1}(a_1)$$ and thus it is clear that $A^G/N_G(A)\cong A^{G_1}/N_{G_1}(A_1)$.
			\end{proof}
			\begin{lemma}
				There is an isomorphism $H^{-1}(G,A)\cong H^{-1}(G_1,A_1)$.
			\end{lemma}
			\begin{proof}
				Let $a=\sum_{i=1}^{s}\sigma_ia_i'$ where $a_i' \in A_1$. If we denote $\alpha = \sum_{i=1}^sa_i'$ we obtain 
				$$N_G(a)=\sum_{j=1}^{s}\sigma_jN_{G_1}(\alpha).$$ 
				Thus it is easy to see that $N_G(a)=0$ if and only if $N_{G_1}(\alpha)=0$. Therefore the map
				$$a\mapsto \alpha$$ is a homomorphism 
				$$\lambda:Ker N_G \rightarrow Ker N_{G_1},$$ which is obviously surjective. We now show that $\lambda$ maps $I_G A$ into $I_{G_1}A_1$. If $\sigma\in G$, then there is a permutation $g$ of the indices $i$ such that 
				$$\sigma\sigma_i=\sigma_{g(i)}\tau_{g(i)}$$ with some $\tau_{g(i)}\in G_1$. Hence
				$$\lambda(\sigma a-a)=\sum_{i=1}^s(\tau_{g(i)}a_i'-a_i')$$ thus proving our assertion. To conclude the proof, it will suffice to show that if $\lambda(a)=0$ then $a\in I_G(A)$. However, if $\alpha=a_1'+\cdots+a_s'=0$ we can write 
				$$a=\sum_{i=1}^s(\sigma_ia_i'-a_i'),$$ and so $a\in I_GA$.
			\end{proof}
		Finally, note that if $L/K$ is a finite Galois extension of fields with Galois group $G=G(L/K)$ then $H^{1}(G,L^\times)=1.$ This result is commonly known as \textit{Hilbert 90} theorem and is of huge importance in local class field theory.
		\subsection{Abstract Class field theory}
		Let $k_0$ be a field and $G=G(k_0^{sep}/k_0)$ be its absolute Galois group. For each finite separable extension $K/k_0$, denote $G_K=G(k_0^{sep}/K)$. Assume that there is a $\hat{\mathbb{Z}}$-extension $\tilde{k}_0/k_0$ and fix the isomorphism $deg:G(\tilde{k}_0/k_0)\cong\hat{\mathbb{Z}}$.
		
		For every finite extension $K/k_0$ we set $\tilde{K}=K\cdot\tilde{k}_0$ and $f_K=[K\cap\tilde{k}_0:k_0].$ Then the map $deg$ induces another isomorphism
		\[deg_K=\frac{1}{f_K} deg:G(\tilde{K}/K)\cong \hat{\mathbb{Z}}.\]	
		\begin{Def} The element $\varphi_K\in G(\tilde{K}/K)$ with $deg_K(\varphi_K)=1$ is called the Frobenius over $K$.	
		\end{Def}
		For every finite extension $L/K$ we set $L^0=L\cap \tilde{K}$, $f_{L/K}=[L^0:K]=\frac{f_L}{f_K}$, $\varphi_{L^0/K}=\varphi_K|_{L^0}$. We have $\varphi_{L}|_{\tilde{K}}=\varphi_K^{f_{L/K}}$. Let $\tilde{L}=L\cdot \tilde{K}$. Then $deg_K:G(\tilde{K}/K)\cong \hat{\mathbb{Z}}$ induces a surjective homomorphism
		$$deg_K:G(\tilde{L}/K)\rightarrow \hat{\mathbb{Z}},$$ and we set $$\phi(\tilde{L}/K)=\{\tilde{\sigma}\in G(\tilde{L}/K): deg_K(\tilde{\sigma})\in \mathbb{Z}_{>0}\}.$$ It can easily be proved that the map
		$$\phi(\tilde{L}/K)\rightarrow G(L/K), \text{ } \tilde{\sigma}\mapsto \tilde{\sigma}|_L,$$ is surjective. If $\sigma=\tilde{\sigma}|_L$, then we call $\tilde{\sigma}$ a Frobenius lift of $\sigma$.
		
		Let $A$ be a multiplicative $G$-module. For each field $k_0\subset K\subset k_0^{sep}$ let $A_K= A^{G_K}.$ If $L/K$ is a finite extension, then $A_K\subset A_L$ and we have the norm map $N_{L/K}:A_L\rightarrow A_K$ defined by
		$$N_{L/K}(a)=\prod_{\sigma}a^\sigma$$ where $\sigma$ runs through a system of right representatives of $G_K/G_L$. Similarly to the usual field norm, it can be checked that this norm map is transitive. If $L/K$ is a Galois extension, then $A_L$ is a $G(L/K)$ module and $A_K=A_L^{G(L/K)}$.
		\begin{Def}
			A henselian valuation of $A_{k_0}$ with respect to $deg$ is a homomorphism 
			$$v:A_{k_0}\rightarrow \hat{\mathbb{Z}}$$ with the properties
			
			i) $v(A_{k_0})=Z\supset\mathbb{Z}$ and $Z/nZ \cong \mathbb{Z}/n\mathbb{Z}$ for all $n\in \mathbb{Z}_{>0}$.
			
			ii) $v(N_{K/{k_0}} A_K)=f_KZ$ for all fields $K.$
		\end{Def}
		Note that the definition does not imply that $v$ is a discrete valuation, and that what we call ``henselian valuation" in this paper is very different, in general, from what is called ``henselian valuation" in valuation theory.
		
		For every field $K$, the henselian valuation $v:A_{k_0}\rightarrow \hat{\mathbb{Z}}$ yields the homomorphism
		$$v_K=\frac{1}{f_K}v\circ N_{K/k_0}:A_K\rightarrow \hat{\mathbb{Z}}$$ with image $Z$.
		We say $\pi_K\in A_K$ with valuation $1$ is a prime of $A_K$.
		
		Furthermore, we now assume thay $G$-module $A$ satisfies the following axiomatic condition:
		
		\begin{cfa}
			For every finite cyclic extension $L/K$ 	
			
			$\#H^i(G(L/K),A)= 
			\begin{cases}
			[L:K] &\text{for } i=0\\
			1 & \text{for } i=-1\\
			\end{cases}$		
		\end{cfa}
		Then for each finite Galois extension $L/K$ 
		we define the reciprocity map $$r_{L/K}:G(L/K)\rightarrow A_K/N_{L/K}A_L$$ given by
		$$r_{L/K}(\sigma): N_{\Sigma/K}(\pi_{\Sigma})\text{ mod } N_{L/K}A_L,$$ where $\Sigma$ is the fixed field of a Frobenius lift $\tilde{\sigma}\in\phi(\tilde{L}/K)$ of $\sigma\in G(L/K)$ and $\pi_\Sigma\in A_\Sigma$ is a prime element. Neukirch proves that this map is well defined and is a homomorphism. 
		
		We define a \textbf{class field theory} for such $G$-module to be a pair of homomorphisms $$(deg:G(\tilde{k}_0/k_0)\rightarrow \hat{\mathbb{Z}},\text{ } v:k_0^\times\rightarrow \hat{\mathbb{Z}})$$ where $deg$ is continuous and bijective and $v$ is a henselian valuation with respect to $deg$. Now we are ready to state the main theorem of Class Field Theory:
		\begin{thm}
			If $L/K$ is a finite Galois extension, then $$r_{L/K}:G(L/K)^{ab}\rightarrow A_K/N_{L/K}$$ is an isomorphism. 
		\end{thm} The inverse map of $r_{L/K}$ yields a surjective homomorphism 
		$(\text{ },L/K):A_K\rightarrow G(L/K)^{ab}$ with kernel $N_{L/K}A_L$. This map is called the \textbf{norm residue symbol} of $L/K$. The norm residue symbol satisfies certain functoriality conditions. In particular we have the commutative diagram below (for $L\subset L'$ and $K\subset K')$. 
		\begin{center}
			\begin{tikzpicture}
			\matrix(m) [matrix of math nodes, row sep=3em, column sep=4em, minimum width=2em]
			{
				A_{K'}&	G(L'/K')\\
				A_K&	G(L/K)\\
			};
			\path[-stealth]
			(m-1-1) edge node [left] {$N_{K'/K}$} (m-2-1)
			(m-1-1) edge node [above] { $($ , $L'/K')$} (m-1-2)
			(m-2-1) edge node [above] { $($ , $L/K)$} (m-2-2)
			(m-1-2) edge node [right] {} (m-2-2);
			\end{tikzpicture}
			\end {center}	
			Passing onto projective limits, the norm residue symbol extends to infinite Galois extensions $L/K$ since we have the above commutative diagram. In the particular case of the extension $\tilde{K}/K$ we have
			\begin{Prop}\label{main}
				$deg_K\circ($ , $\tilde{K}/K)=v_K$. In particular $(a,\tilde{K}/K)=\varphi_K^{v_K(a)}$.
			\end{Prop}
			An another important feature of classical class field theory is the Existence Theorem. For every field $K$ we introduce a topology in $A_K$ as follows. For each $a\in A_K$ we take the cosets $aN_{L/K}A_L$ as a basis of open neighbourhoods of $a$, where $L/K$ runs through all finite Galois extensions of $K$. We call this topology the norm topology of $A_K$. Note that $v_K$ and $N_{L/K}$ is continuous if $L/K$ is a finite extension. Then we can classify the finite abelian extensions.
			\begin{thm}
				The map $l\mapsto \mathscr{N}_L=N_{L/K}A_K$ yields a $1-1$ correspondence between the finite abelian extensions $L/K$ and the open subgroups $\mathscr{N}$ of $A_K$. Moreover 
				$$L_1\subset L_2 \iff \mathscr{N}_{L_1}\supset\mathscr{N}_{L_2}, \text{ }\mathscr{N}_{L_1\cdot L_2}=\mathscr{N}_{L_1}\cap \mathscr{N}_{L_2},\text{ } \mathscr{N}_{L_1\cap L_2}=\mathscr{N}_{L_1}\cdot\mathscr{N}_{L_2}.$$
			\end{thm}
			If $L$ is associated to the open subgroup $\mathscr{N}$ of $A_K$, then $L$ is called the \textit{class field} of $\mathscr{N}$. 
			\subsection{Local Class Field Theory}			
			We shall briefly recall results from $\cite{NANT}$ without proof. 
			
			Let $k_0=\mathbb{F}_p((T))$ or $\mathbb{Q}_p$ and $A=(k_0^{sep})^\times$. Thus $A_K^\times=K^\times$ for any finite extension $K/k_0$. Let $v$ be the usual normalised discrete valuation of $k_0$ and $\tilde{k}_0= k_0^{ur}$. Note that in the characteristic $p$ case $k_0^{ur}= \mathbb{F}_p^{sep}((T))$ and in the characteristic $0$ case $k_0^{ur}$ is obtained by adjoining all roots of unity coprime to $p$. Then there is a canonical isomorphism $deg:G(k_0^{ur}/k_0)\cong \hat{\mathbb{Z}}$. Let $K$ be a finite extension of $k_0$. Then one can prove the following:
			\begin{thm}
				The pair $(deg:G(k_0^{ur}/k_0)\rightarrow \hat{\mathbb{Z}},$ $v:k_0^\times\rightarrow \hat{\mathbb{Z}}$) is a class field theory and $A=(k_0^{sep})^\times$ satisfies the Class Field Theory Axiom. Therefore, for each Galois extension $L/K$ of local fields we obtain a canonical isomorphism
				$$r_{L/K}:G(L/K)^{ab}\rightarrow K^\times/N_{L/K} L^\times.$$ Also the Existence Theorem holds.
			\end{thm}
			Note that the only part which requires work is the proof of the class field theory axiom which can be found in $\cite{NCFT}$ and \cite{NANT}. In fact, by Hilbert 90, it suffices to calculate the Herbrand quotient only. By the exact sequence $$0\rightarrow U_L\rightarrow L^\times\overset{v_L}{\longrightarrow}\mathbb{Z}\rightarrow 0$$ it suffices to prove the following:
		
			\begin{lemma}\label{hunit}
				Let $L/K$ be a finite extension of local field and $U_L$ be unit group of $L$. Then we have $h(G(L/K),U_L)=1$. 
			\end{lemma}
			\begin{proof}
				See \S5 of \cite{NANT}.
			\end{proof}
			As before, the inverse of $r_{L/K}$ gives the local norm residue symbol
			$$(\text { },L/K):K^\times\rightarrow G(L/K)^{ab}$$ with kernel $N_{L/K}L^\times$.
			\newpage
				\section{Characteristic $p$ global class field theory}
				We shall prove characteristic $p$ global class field theory in the spirit of Neukirch's explicit approach.
				\subsection{Preliminary results for Global Function Fields}
				Here we recall some standard facts about function fields, most of which can be found in \cite{AFFC}. This section is for those who have little to no experience with characteristic $p$ global field but is familiar with number fields.
				
				\subsubsection{General Function field}
				We start with an algebraic definition of function field of one variable over an arbitrary field $F$.
				\begin{Def}
					Let $F$ be a field. An algebraic function field $K/F$ of one variable over $F$ is an extension field $K\supset F$ such that $K$ is a finite algebraic extensions of $F(x)$ for some element $x \in K$ which is transcendental over $F$. If $K=F(x)$ for some $x\in K$ then we call $K/F$ a rational function field.
				\end{Def}
				Clearly the set $F':=\{z \in K | z$ is algebraic over $F\}$ is a subfield of $K$. We call $F'$ the field of constants of $K/F$. It is clear that $K$ is a function field over $F'$. In this text we will always assume $F=F'.$
				\begin{Def}
					A  discrete valuation of $K/F$ is a map $v: K^\times\rightarrow \mathbb{R}$ satisfying the following conditions:
					
					$1.$ $v(yz)=v(y)+v(z)$ for all $y,z \in K$.\\
					$2.$ $v(y+z)\geq$ min$\{v(y)+v(z)\}$ for all $y,z \in K$.\\
					$3.$ $v(\alpha)=0$ for all $\alpha \in F$.\\
					$4.$ $v(K^\times)$ is a non-zero discrete subset of $\mathbb{R}$. 
				\end{Def}
				We can define an equivalence relation on the set of discrete valuations:
				$v \sim v'$ if there exists a constant $c>0$ such that $v=cv'$. An equivalent class of discrete valuations of $K/F$ is called a \textbf{Place} of $K/F$; the set of all places of $K$ is denoted by $\mathbb{P}_K$. We also say a discrete valuation $v$ is normalised if $v(K^\times)=\mathbb{Z}$. Note that for a discrete valuation $v$ of $K/F$, $v(K^\times)$ is a discrete subgroup of $(\mathbb{R},+)$ so $v(K^\times)=c\mathbb{Z}$ for some $c\in \mathbb{R}_+$. So clearly we can identify places of $K/F$ and normalised valuations of $K/F$. For each $P$ we write the corresponding valuation $v_P$. 
				
				For a place $P$ we set $O_P=\{z\in K:v_P(z)\geq 0\}$. This is a discrete valuation ring with a unique maximum ideal $m_P:=\{z\in K:v_P(z)\geq 1\}$. The quotient $O_P/m_P=:F_P$ is a finite extension of $F$. So we may define the Degree of a place by $$\mathrm{Deg}(P)=[F_P:F].$$ We say a place $P$ is rational if Deg$(P)=1$.
						
				\subsubsection{Global Function Field}
				Now we restrict our attention to the case of global function field. Thus $F=\mathbb{F}_{p^n}$ and $K$ is a function field over $F$. For each place $P$ let $K_P$ be a completion of $K$ with respect to the topology induced by the valuation $v_P$. Then $K_P$ is a local field of characteristic $p$. We note that the residue field of $K_P$ is $F_P$ and $K_P=F_P((t_P))$ where $t_P$ is a prime of $K_P$. We write $\mathcal{O}_P$ for the ring of integers of $K_P$ and $U_P=\mathcal{O}_P^\times$ for its units.
				
				\begin{Def}
					The divisor group of $K/F$ is defined as the (additively written) free abelian group which is generated by the places of $K/F$; it is denoted by $\mathrm{Div}(K)$. The elements of $\mathrm{Div}(K)$ are called divisors of $K/F$. In other words, a divisor is a formal sum
					$$D=\sum_{P\in \mathbb{P}_K} n_PP\quad \text{ with }$$ where $n_P\in \mathbb{Z}$ and almost all $n_P=0$. The support of $D$ is defined as
					$$\mathrm{supp}(D):=\{P\in\mathbb{P}_K:n_P\neq0\}.$$ It will often be found convenient to write
					$$D=\sum_{P\in S}n_PP.$$ where $S\in \mathbb{P}_F$ is a finite set with $S \supset \ \mathrm{supp}D$. A divisor of the form $D=P$ for some place $P$ is called a prime divisor.	\\The degree of a divisor is defined as $$\mathrm{Deg}D:=\sum_{P\in\mathbb{P}_K}v_P(D)\mathrm{Deg}P,$$ and this yields a homomorphism $\mathrm{Deg}:\mathrm{Div}(K)\rightarrow\mathbb{Z}$.
				\end{Def}
				For each $x\in K^\times$ we can define 
				$(x):=\sum_P v_P(x)P$. $v_P(x)$ is zero for almost all $P$ so $(x)$ is a divisor. Also it can be proved that Deg$((x))=:$ Deg$(x)=0$ for all $x\in K$. Thus we can define the following:
				\begin{Def}
					The set of divisors
					$$\mathrm{Princ}(K):=\{(x):0\neq x \in K\}$$ is called the group of principal divisors of $F/K$. This is a subgroup of $\mathrm{Div}(K)$, since for $x,y\in K^\times$, $(xy)=(x)+(y)$. The factor group
					$$Cl(K):=\mathrm{Princ}(K)/\mathrm{Div}(K)$$ is called the divisor class group of $K/F$.
				\end{Def}
				By the remark above, Deg$:\mathrm{Div}(K)\rightarrow \mathbb{Z}$ induces Deg$:Cl(K)\rightarrow \mathbb{Z}$.
				
				\begin{Prop}\label{divisor class}
					Let $Cl^0(K)$ be a subgroup of $Cl(K)$ which consists of element which has degree zero. Then $Cl^0(K)$ is a finite group.
				\end{Prop}
				\begin{proof}
					This standard result can be found in many texts, including \cite[\S5]{AFFC}.
				\end{proof}
				\subsubsection{Useful results on Global Function fields}
				Here we state a few results essential for the proof of class field theory axiom. Most of it is taken from $\S13,\S17$ of \cite{Artin}. Results which we directly need will be stated without proof. 
				
				Firstly we define the derivative on a local field (of characteristic $p$) in the following way: let $t_P$ be a prime element of a local field $K_P$ and $x=\sum_{i=m}^\infty a_i t_P^i$. Then we define $\frac{dx}{dt_P}=\sum_{i=m}^\infty ia_{i}t_P^{i-1}$. Given such expansion we define a function $a_i^{t_P}:K_P^\times\rightarrow F_P$ by $a_i^{t_P}(x)=a_i$. Clearly this depends on the choice of the prime element $t_P$. Note if $t_P'$ is another prime then 
				$$ \frac{dy}{dt_P}=\frac{dy}{dt_P'}\frac{dt_P'}{dt_P} $$
				Then in general,  for arbitrary $x,y \in K_P$ where $\frac{dy}{dt_P}\neq 0$,  we can define $\frac{dx}{dy}$, by the formula
				$\frac{dx}{dy}=\frac{dx}{dt}/\frac{dy}{dt}.$ This is independent of the choice of the prime $t_P$ by the above `chain rule'. 
				
				We also define another map $res$ on a local field as following:
				Let $xdy$ be a local differential. Then we define residue of a differential form
				$$res^{t_P}(xdy):=a_{-1}^{t_P}(x\frac{dy}{dt_P}).$$ However this function does not depend on the choice of prime $t_P$ so we may just write $res$ instead of $res^{t_P}$. We may also write $res_P$ for the residue function on $K_P$.
				\begin{Def}
					Let $x\in K$ which is transcendental over $F$. We say $x$ is a separating element if $K/F(x)$ is a finite separable extension.
				\end{Def}
				Now we state a Lemma which allows us to describe a global function field in a neater way:
				\begin{lemma}\label{sep}
						For each place $P$, there exists a separating element $t\in K$ such that $t$ is a prime of $K_P$.
				\end{lemma}
				
				\begin{proof}
					Let $x$ be a separating element. By replacing $x$ by $1/x$, $x+1$ or $1/x+1$, we can ensure that $v_P(x)=0$. Let $\tau\in k$ be any prime of $P$. If $\tau$ is a separating element then we are done. If not then $d\tau/dx=0$. Set $t=\tau x$. Then $t$ is also a prime of $P$ and a separating element as
					$$\frac{dt}{dx}=\tau\neq0$$ by \cite[\S4]{AFFC}.
				\end{proof}
				Then we get the following result.
				\begin{lemma}\label{dt/dp}
					Let $t$ be a separating element for $K/F$ which also is a prime for a place $Q$. For each $P\in \mathbb{P}_K$, let $t_P\in K_P$ be a prime. Then $\frac{dt}{dt_P}$ is in $U_P$ for almost all $P$. Consequently we have a well defined element $(\frac{dt}{dt_P})$ of $I_K$.
				\end{lemma}
				\begin{proof}
					For a rational function field, this can be verified with direct calculation. For a general function field it follows from \cite[\S17]{Artin}.					
				\end{proof}	
				
				Let $\mathbb{A}_K:=\prod_{P}'K_P$ be the ring of Adeles of $K$. We define $Lin_{cts}^*(K)$ be space of continuous $F$-linear map $\mu$ of the ring $\mathbb{A}_K$ into $F$ such that $\mu(a)=0$ for $a\in K$. Then $Lin_{cts}^*(K)$ forms a vector space over $K$.  Note that in the original text (\cite{Artin}) the space $Lin_{cts}^*(K)$ is called differential but in this paper we choose not to use this notation since it could be confused with derivative and differential forms.
				
				Now we state an important property of $Lin_{cts}^*(K)$.
				\begin{Prop}\label{dc}
				The space $Lin_{cts}^*(K)$ forms a one-dimensional vector space over $K$. Also if $\mu(\alpha)=0$ for all $\mu\in Lin_{cts}^*(K)$, then $\alpha \in K$.
				\end{Prop}	
				\begin{proof}
					See \cite[\S13]{Artin}.
				\end{proof}
				Let $Q$ be a place and let $t$ be a separable element which is also a prime for $Q$. For each place $P$ let $Tr_P$ be the trace map between $F_P$ and $F$. Then we define a map $\lambda:\mathbb{A}_k\rightarrow F$ given by
				$$\lambda(\alpha)=\sum_PTr_P(res_P(\alpha_Pdt)).$$ This map is indeed in $Lin_{cts}^*(K)$; it is clearly additive and $k_0$ homogeneous. If $|\alpha_{P}|_{P}$ is small for all places where $t$ has poles and if $\alpha_P$ is a unit at all other places, then $\lambda(\alpha)=0$, since each residue is $0$. This shows continuity of $\lambda$. Vanishing at $K$ requires some work which we leave readers to check \cite{Artin} $\S17$.
				
				As the action of $K$ on $\lambda \in Lin_{cts}^*(K)$ is given by $x\cdot\lambda(\alpha)=\lambda(x\alpha)$, Proposition $\ref{dc}$ implies that if $\lambda(x\alpha)=0$ for all $x\in K$, then $\alpha\in K.$
			
				%\begin{thm}
				%Let $\mu$ be a differential of $k$. Then we define $\mu_P(\alpha)$ to mean $\mu(\sigma_P)$. Then $$\mu(\alpha)=\sum_P\mu_P(\sigma)$$.
				%\end{thm}
				%	Again check $\cite{Artin}$ for proof	
				%\begin{cor}
				%	Let $\mu$ and $\mu'$ be differentials. Then $\mu_1=\mu_2$ if and only if $\mu_P=\mu'_P$ for one place $P$.
				%\end{cor}
				\subsection{Idele Class group}
				Let $k_0= \mathbb{F}_p(T)$ and let $K$ be a finite extension of $k_0$. Let $q$ be the size of the constant field of $K$. Then we define $I_K$, group of Ideles over $K$ to be
				$\prod_P'K_P^\times$, the restricted product of the multiplicative group of the completion of $K$ at each place $P$ with respect to their units $U_P$. $K^\times$ is naturally diagonally embedded into $I_K$ and therefore we can define $C_K=I_K/K^\times$, the Idele class group. It is well known that Idele class group for global fields satisfy Galois descent condition (For $L/K$ a Galois extension, $C_L^{G(L/K)}=C_K)$. 
				
				We define an Idelic norm $|\cdot|_K:I_K\rightarrow \mathbb{R}_+$ by $$|a|_K=\prod_P q_P^{v_P(a)}$$ where $q_P=q^{deg(P)}$, in another words, the size of the residue field of $K_P$. This is a homomorphism where the image is $q^\mathbb{Z}\subset \mathbb{R}_+$. We denote the kernel of this map by $I_K^0$. By the product formula $K^\times \in I_K^0$ and therefore we can define subgroup $C_K^0$ of $C_K$ by $I_K^0/K^\times.$ In the exactly same way as the number field case it can be shown that $C_K^0$ is compact.
				\begin{Prop}\label{Decomp}
					$C_K = C_K^0 \times \Gamma_K$ where $\Gamma_K \cong \mathbb{Z}$. 
				\end{Prop}	For the proof of this Proposition we introduce a new map. For every place $P$ of $k$ we have a canonical injection
				$$[\text{ }]_P:K_P^\times\rightarrow C_K$$ which maps $a_P\in K_P^\times$  to the class of the Idele
				$[a_P]_P=(\cdots,1,1,1,a_P,1,1,1,\cdots).$ 
				\begin{proof}
					Let $P$ be a rational place of $K$ and let $t_P$ be a prime element of $K_P$. Let $\Gamma_K$ be the image of the subgroup generated by $[t_P]_P$ in $C_K$. Then clearly $\Gamma_K \cong \mathbb{Z}$. By construction $\Gamma_K$ is mapped isomorphically onto $q^\mathbb{Z}\subset \mathbb{R}_+$ by the Idelic norm and hence we obtain $C_K=C_K^0\times \Gamma_K$.
				\end{proof}		
				\begin{Prop}\label{norm}
					Let $L/K$ be a finite extension and $\alpha \in I_L$, then the local components of $N_{L/K}(\alpha)$ is given by
					$$N_{L/K}(\alpha)_P=\prod_{\mathfrak{P}|P}N_{L_\mathfrak{P}/K_P}(\alpha_\mathfrak{P}).$$
				\end{Prop}
				\begin{proof}
					Same as the number field case. See \S6.2 of $\cite{NANT}$.
				\end{proof}
				\subsection{Class field theory axiom}	
				
				We plan to prove the class field theory axiom for characteristic $p$ global field $K$, taking Idele class group $C_K$ as the central object. 
				
				First we plan to prove two inequalities, which are classically known as the First and the Second inequality. The proof of the Second inequality in the case where the degree of the extension divides $p=char(K)$ requires some lengthy argument to prove. So the readers who would like to assume the inequalities and see the original work, may skip to the next chapter. The proof was taken from $\cite{AT}$.
				\subsubsection{The First inequality}
				We state one side of inequality which we shall refer to as the First Inequality.
				\begin{thm}\label{FIQ}
					Let $L/K$ be a cyclic extension of degree $n$ with Galois group $G=G(L/K)$. Then $h(G(L/K),C_L)=n$. In particular $$\#{H}^0(G,C_L)=(C_k:N_{L/K}C_L)\geq n.$$
				\end{thm}
				Note that this proof is a lot easier than corresponding result in characteristic $0$ as global function fields do not have an archimedian place. For this proof only we write $h(A):=h(G(L/K),A)$ for each $G(L/K)$-module $A$.
				\begin{proof}
					Let $U=\prod_\mathfrak{P} U_\mathfrak{P}$ be the unit Ideles of $L$. Clearly $U\subset I_L^0$ so $UL^\times\subset I_L^0$. The multiplicity of Herbrand quotient gives
					$$h(I_L/L^\times)=h(I_L/I_L^0)h(I_L^0/UL^\times)h(UL^\times/L^\times).$$ 
					Using Proposition \ref{Decomp} and the Third Isomorphism Theorem we obtain $h(I_L/I_L^0)=h(\mathbb{Z})=n$. We recall that Proposition \ref{divisor class} says the number of divisor classes of degree zero is finite. As $\mathrm{Cl}^0(L)\cong I_L^0/UL^\times$ is a finite group, $h(I_L^0/UL)=1$.
					
					$UL^\times/L^\times$ is isomorphic to $U/(U\cap L^\times)$ as a $G(L/K)$-module and hence 
					$$h(UL^\times/L^\times)=h(U/(U\cap L^\times))=h(U)h(U\cap L^\times)^{-1}.$$ $U\cap L^\times$ is the multiplicative group of the constant field of $L$ so it is a finite group. Hence $h(U\cap L^\times)=1$. We can express $U$ as a direct product
					$$U=\prod_{P}(\prod_{\mathfrak{P}|P}U_\mathfrak{P}).$$ Let $U_P$ denote any of $U_\mathfrak{P}$ for $\mathfrak{P} | P$ then we have for $i=0, -1$,
					$$H^i(G,\prod_{\mathfrak{P}|P}U_\mathfrak{P})\cong H^i(G_P, U_P)$$ for $i=0,-1$ by the semilocal theory (Lemma 1,2) and hence $$H^i(G(L/K),U)\cong\prod_P H^i(G_P,U_P)$$

					By Lemma \ref{hunit}, $h(G_P,U_P)=1$ for all $P$. Therefore $h(U)=\prod_P h(G_P,U_P)=1$ and $h(C_L)=n.$ 
				\end{proof}
				Now we state two consequences of the First inequality. These will play important role for the proof of the Second inequality.
				\begin{cor}\label{fc1}
					Let $L/K$ be a Galois extension of a global field. If almost all places split completely in $L$, then $L=K$. Moreover if $L/K$ is cyclic of prime degree $p^r$, then there are infinitely many places of $K$ which do not split in $L$.
				\end{cor}
				\begin{proof}
					See \cite{NANT}. (The proof in the text is given for the number field case but as we have Theorem \ref{FIQ}, the same proof can be used for function fields.)
				\end{proof}
				\begin{cor}\label{fc2}
					Let $K_1$, $\cdots$, $K_r/K$ be $r$ cyclic extensions of prime degree $p$ which are mutually disjoint over $K$. Then there exists infinitely many places $P$ that split completely in $K_i$ $(i>1)$ and do not split in $K_1$. 
				\end{cor}
				\begin{proof}
					Let $L=K_1\cdots K_r$ be the compositum of all $K_j$ for all $1\leq j \leq r$. Then $L/(K_2\cdots K_r)$ is cyclic. Let $\mathfrak{P}$ be a place in $(K_2\cdots K_r)$ which does not split in $L$ and which divides a place $P$ of $K$ which is unramified in $L$. Since almost all places of $(K_2\cdots K_r)$ is unramified in $L$, such $\mathfrak{P}$ exists by the above Corollary. Then $L_\mathfrak{P}/(K_2\cdots K_r)_\mathfrak{P}$ is cyclic of degree $p$. But $L_\mathfrak{P}/K_P$ is also cyclic as $P$ is unramified in $K$. The Galois group of $L/K$ is isomorphic to $(\mathbb{Z}/p\mathbb{Z})^r$ and $L_\mathfrak{P}/K_P$ is a cyclic subgroup. This means $[L_\mathfrak{P}:K_P]\leq p$. Together with $[L_\mathfrak{P}:(K_2\cdots K_r)_\mathfrak{P}]=p$, this shows $(K_2\cdots K_r)_\mathfrak{P}=K_P$. It follows that $P$ splits completely in $(K_2\cdots K_r)$. It does not split in $K_1$ as otherwise $L_\mathfrak{P}=K_P$ which is not the case.
				\end{proof}
				\subsubsection{Second Inequality}
				Let $L/K$ be a cyclic extension of degree $n$ with Galois group $G$ as before. The purpose of this section is to prove the following:
				\begin{Prop}
					$\#H^0(G,C_L)| n$. In particular $\#H^0(G,C_L) \leq n$ and hence equal to $n$.
				\end{Prop}	 	
				This inequality is classically known as the Second inequality. It is possible to perform reduction steps to show that it suffices to prove the inequality in the case of cyclic extension of prime degree. 
				\begin{lemma}
					If $L/K$ is a finite extension of degree $m$ then $[C_K:N_{L/K}C_L]$ is finite and divides a power of $m$.
				\end{lemma}
				\begin{proof}
					To prove the finiteness of norm index we can assume $L/K$ is normal using transitivity of the norm map. Also we write $N=N_{L/K}$ for the proof of this Lemma only. 
					
					Let $S$ be a finite subset of $\mathbb{P}_K$ including all ramified places, and enough places such that $I_K=KI_K^S$, $I_L=LI_L^S$. See definition of $I_K^S$ and $I_L^S$ in \cite{NANT}. Then $KNI_L=K\cdot NL\cdot NI_L^S=KNI_L^S$. So we obtain
					$$[I_K:KNI_L]=[KI_K^S:KNI_L^S]\leq[I_K^S:NI_L^S]=\#H^0(G,I_L^S)= \prod_{P\in S} n_P$$ where $n_P=\#H^0(G_P,K_P^\times)$. The last equality follows from semilocal theory. This proves the finiteness.
					
					The norm index divides a power of the degree because for any $a\in C_K$, $a^m\in N_{L/K}C_L$ for any $L\supset K$.
				\end{proof}
				\begin{lemma}\label{reduction}
					Let $K\subset L' \subset L$ be two finite extensions. Then
					
					1. $[C_K:N_{L'/K}(C_{L'})]$ divides $[C_K:N_{L/K}(C_L)]$. 	\\		
					2. $[C_K:N_{L/K}(C_L)]$ divides $[C_K:N_{L'/K}(C_{L'})]\cdot [C_{L'}:N_{L/L'}(C_L)]$
					
					Consequently if the Second Inequality holds for $L/L'$ and $L'/K$ then it also holds for $L/K$.
				\end{lemma}
				\begin{proof}
					We note that 
					$$[C_K:N_{L/K}(C_L)]=[C_K:N_{L'/K}(C_{L'})][N_{L'/K}(C_{L'}):N_{L'/K}(N_{L/L'}(C_L))].$$ This proves $1$. Considering the homomorphism $C_{L'}\rightarrow N_{L'/K}(C_{L'})$ we can see that $ [N_{L'/K}(C_{L'}):N_{L'/K}(N_{L/L'}(C_L))]$ divides $[C_{L'}:N_{L/L'}(C_L)]$.
				\end{proof} 		
				From this we obtain:
				\begin{cor}
					If the Second inequality holds in all cyclic extensions of prime degree $l$, then it holds in all Galois extensions. 
				\end{cor}
				\begin{proof}
					Let $L/K$ be Galois and let $l$ be a prime. Let $L'$ be the fixed field of an $l$-Sylow subgroup of the Galois group $G=G(L/K)$. $L/L'$ is a tower of cyclic fields of degree $l$ and by Lemma \ref{reduction} we may assume that the inequality holds in $L/L'$.
					
					$[C_{L'}:N_{L/L'}(C_L)]$ divides $[L:L']$. On the other hand, $[C_K:N_{L'/K}(C_{L'})]$ divides a power of $[L':K]$, and is therefore prime to $l$. From the fact that $[C_K:N_{L/K}(C_L)]$ divides $[C_K:N_{L'/K}(C_{L'})]\cdot [C_{L'}:N_{L/L'}(C_L)]$ it follows now that the $l$-contribution to $[C_K:N_{L/K}(C_K)]$ divides $[L:L']$ and consequently $[L:K]$. Then the inequality for $L/K$ follows.
				\end{proof}
				%If $[L:k]=l$ where $l$ is a prime not equal to the characteristic of $k$, there is an explicit description of $C_k/N_{L/k}C_l$ in \cite{ANTC} p180-184 and shows that in that case $[C_k:N_{L/k}C_L]=[L:k]$. The approach is basically the same as Neukirch's approach in $\cite{NANT}$ but there have been some minor changes to include the function field case.
				
				Using the above corollary we now have reduced to the case where $[L:K]=l$ is a prime. If $l \neq p= char(K)$, Neukirch's proof in number field case generalises naturally. This lengthy argument is easy to find in many texts (for instance in \cite[p.~180$-$184]
				{ANTC}) so we shall omit the proof.
				
				Now we consider extensions of degree $p$. For this we first recall Artin-Schreier Theory:
				\begin{thm}\label{AS}
					Let $K$ be a field of characteristic $p>0$. Then there is one to one correspondence between the additive subgroups $\Delta$ of $K$ which contain $\mathscr{P}K,$ and abelian extensions $L/K$ with groups of exponent $p$. The correspondence is given by
					$$\Delta \rightarrow L_{\Delta}=K(\mathscr{P}^{-1}\Delta).$$
				\end{thm}
				This is analogue to Kummer theory which plays an important role in the characteristic $0$ and $[L:K]\neq p$ case. Note we can also describe any cyclic extension of degree $p^n$ of a characteristic $p$ field in a similar manner by using Witt vectors. For this see \cite{Witt}.
				
				Now let $K$ be a global function field with constant field $F=\mathbb{F}_{p^r}$. For each place $P$ of $K$ let $K_P$ be the usual local field associated with $P$ and let $F_P$ denote its residue field. Then $F_P$ naturally contains $F$ and $K_P=F_P((t))$ where $t$ is a prime element of $K_P$ which is also in $K$. By Lemma \ref{sep} we can also assume $t$ is separating.
				
				Let $Tr=Tr_{F_p/\mathbb{F}_p}$ and for $x$, $y \in K_P$ let $res(xdy)$ be the residue of a local differential $xdy$ with respect to $t$. 
				
				For $x\in K_P$ and $y \in K_P^\times$ we define 
				$$\psi_P(x,y)=Tr(res(x\frac{dy}{y}))=Tr(res(x\frac{\frac{dy}{dt}}{y}))dt$$
				Then we easily verify that $\psi_P(x_1+x_2,y)=\psi_P(x_1,y)+\psi_P(x_2,y)$ and $\psi_P(x,y_1y_2)=\psi_P(x,y_1)+\psi_P(x,y_2)$ and therefore is a pairing of the two groups $K_P$ and $K_P^\times$ into additive group $\mathbb{F}_p$. This pairing is continuous since if $x\in \mathcal{O}_{P}$ and $y\in \mathcal{O}_{P}^\times$ then $res(x\frac{dy}{y})=0$.
				\begin{lemma}\label{local pairing}
					Let $x$ be in $\mathcal{O}_{P}$. Then $\psi_P(x,y)=0$ if and only if either $p | v_P(y)$ or $x=\mathscr{P}z$ for some $z\in K_P$. 
				\end{lemma}	
				\begin{proof} 
				Let $x=\sum_{i=0}^{\infty} a_it^i$ with $a_i\in F$, and $y=t^nu$ where $u$ is a unit of $K_P^\times$. $\psi_P(x,u)=0$ and consequently
				$$\psi_P(x,y)=nTr(res(x\frac{dt}{t}))=nTr(a_0).$$
				$\psi_P(x,y)=0$ is therefore equivalent with $p|n$ or $Tr(a_0)=0$ and we have to show that $Tr(a_0)=0$ is equivalent with $x=\mathscr{P}z$.
				
				Suppose $x=z^p-z$. It is easy to prove $z\in \mathcal{O}_{P}$ by considering valuation of $z^p-z$. If $z=\sum_{i=0}^{\infty} b_it^i$ then we have $a_0=b_0^p-b_0$. Recalling raising to the power of $p$ is the generating element for $F_p/\mathbb{F}_p$, we view it as $a_0=(\sigma-1)b$ and as $H^{-1}(G(L/K),K)=0$ for any cyclic extension $L/K$ we obtain $Tr(a_0)=0$. Conversely suppose $Tr(a_0)=0$. Then $a_0=(\sigma-1)b=\mathscr{P}b$ for some $b \in F_P$. Let $x_1=x-a_0$ and put $z_1=-(x_1+x_1^p+x_1^{p^2}+\cdots).$ This series converges in $k_P$ and therefore we get $x_1=z_1^p-z_1-\mathscr{P}z_1$ and consequently
				$$x=\mathscr{P}(b+z_1).$$
				\end{proof}
				Now break the $Tr$ map into trace $Tr_P=Tr_{F_P/F}$ and followed by trace $Tr'=Tr_{F/\mathbb{F}_p}$. Choose a place $Q$ of $K$ and let $t$ be a prime of $K_Q$ separating in $K$. For $\alpha\in \mathbb{A}_K$ let
				$$\lambda(\alpha)=\sum_PTr_P(res_P(\alpha_P dt))$$ as before. Put now 
				$$f(\alpha)dt=Tr'(\lambda(\alpha)).$$
				\begin{lemma}
					Let $\alpha \in \mathbb{A}_K$ and suppose that $f(\alpha x)dt=0$ for all $x \in K$. Then $\alpha\in K$.
				\end{lemma}
				\begin{proof}
					We may replace $x$ by $cx$ for any $c\in F$ and obtain $Tr'(c\cdot\lambda(\alpha x))=0$. This means $\lambda(\alpha x)= 0$ as if it were not the case then we have $Tr'(F)=0$. This is not true as $\mathbb{F}_p$ is perfect. Then by Proposition $\ref{dc}$  that $\alpha \in K$.
				\end{proof}
				We recall that for almost all $P$ the derivatives $\frac{dt}{dt_P}$ is in $U_P$ where $t_P$ is any prime element at place $P$. Let $\alpha\in I_K$. Then we can form an Adele $\frac{1}{\alpha}\frac{d\alpha}{dt}$ with local components 
				$$\frac{1}{\alpha_P}\frac{d\alpha_P}{dt}=\frac{1}{\alpha_P}\frac{d\alpha_p}{dt_P}(\frac{dt}{dt_P})^{-1}.$$ Now we define the following pairing on $K\times I_K$ into $\mathbb{F}_p$:
				$$\psi(x,\alpha)=f(x)dt=Tr'\sum_PTr_P(res_P(x\frac{1}{\alpha_P}\frac{d\alpha_P}{dt}dt))=\sum_PTr(res_P(x\frac{d\alpha_P}{\alpha_P}))=\sum_P\psi_P(x,\alpha_P).$$ Last expression shows that $\psi$ is indeed a pairing. For continuity nothing needs to be shown for the discrete $K$. If $\alpha_P$ is close to $1$ then $\frac{d\alpha_P}{dt_P}$ is close to $0$. Now we describe the kernel of this pairing. For a pairing $\eta:A \times B \rightarrow \mathbb{F}_P$, kernel of $A$ is defined as the set of elements $a\in A$ such that $\eta(a,b)=0$ for all $b\in B$. Kernel of $B$ is also defined in a similar way.
				\begin{Prop}\label{pairing}
					The kernel of $I_K$ in our pairing $\psi$ is precisely $K^\times\cdot I_K^p$ and the kernel of $K$ is the additive group $\mathscr{P}K$.
				\end{Prop}
				We additionally note that every element of $K^\times$ has order $p$ and that $I_K/K^\times\cdot I_K^p\cong C_K/C_K^p$ is compact. Therefore the duality theory can be used. 
				\begin{proof}
					See Appendix.
				\end{proof}	 	 		 		
				Let $K^{sep,p}=K(\mathscr{P}^{-1}K)$ be the maximal $p$-extension of $K$. By Theorem $\ref{AS}$, $K^{sep,p}$ is associated with $A=\mathscr{P}^{-1}K$.  $A$ is paired with the Galois group $G_p=G(K^{sep,p}/K)$ by $(a,\sigma)=(\sigma-1)a$. We have mapped $A$ into $K$ by the map $a \mapsto \mathscr{P}a$. The image of $A$ is the whole field $K$ and we may view this pairing also as a pairing of $K$ and $G_p$ into $\mathbb{F}_p$ where
				$$[x,\sigma]=(\sigma-1)(\mathscr{P}^{-1}x),$$ $x\in K,$ $\sigma \in G_p$. The kernel of $K$ is now $\mathscr{P}K$, the same as in the pairing $\psi(x,G_p);$ that of $G_p$ is $1$. Under this pairing the group $G_p$ is naturally isomorphic to the character group of $K/\mathscr{P}K$ (where $K/\mathscr{P}K$ has the discrete topology). In the pairing $\psi(x,G_p)$ the group $I_K/K^\times I_K^p$ is naturally isomorphic to the same character group. It follows that one has a natural isomorphism between $G_p$ and $I_K/K^\times I_K^p$. If $\sigma$ is the image of the Idele $a$ under the map then we should have $\psi(x,a)=[x,\sigma]=(\sigma-1)(\mathscr{P}^{-1}x)$. We denote this map by $\sigma(a)$ and the previous equation describes for us just the effect of $\sigma(a)$ on the generators of $K^{sep,p}$, a description which determines $\sigma(a)$ completely:
				$$\sigma(a):\mathscr{P}^{-1}x\mapsto\mathscr{P}^{-1}x+\psi(x,a).$$ This map is actually a surjective homomorphism of $C_K$ into $G_p$ with kernel $C_K^p$. We therefore have:
				\begin{thm}\label{dual}
					The map above is a topological isomorphism between $I_K/K\cdot I_K^p$ (or $C_K/C_K^p$) and the Galois group of the maximal $p$-extension $K^{sep,p}$.
				\end{thm} 
				We finally have all the ingredients required for the proof.	
				\begin{proof}[Proof of the Second inequality in the case $n=p$]
					Let $L/K$ be a cyclic extensions of degree $p$. As $L$ is a subfield of $K^{sep,p}$, there is an open subgroup $H$ of $G_p$ of index $p$ such that $(K^{sep,p})^{H}=L$. The norm group $K\cdot N_{L/K}(I_L)$ is open and of finite index in $I_K$. According to Theorem \ref{FIQ} the index is at least $p$. The image of this group under the isomorphism of Theorem \ref{dual} is an open subgroup $H'$ of $G_p$. The index of $H'$ in $G_p$ is finite and at least $p$. $H'$ determines a certain subfield $E/K$ of $K^{sep,p}/k$ which is the compositum of cyclic fields of prime degree and each of these fields is left fixed by $H'$. We shall prove that for cyclic extension $K'/K$ of degree $p$ which is different from $L$, there is an element of $H'$ which does not leave every element in $K'$ fixed. This will show that $H'=H$ and prove the second inequality.
					
					By Corollary \ref{fc2} there are infinitely places $P$ of $K$ which splits completely in $L$ and remains prime in $K'$. If $K'=K(\mathscr{P}^{-1}x)$ we can therefore select $P$ in such a way that $x$ does not have a pole at $P$. Let $\pi_P$ be a prime of $K_P$ and consider the Idele $[\pi_P]_P$. It is a norm from $L$ since $P$ splits in $L$ and hence $\sigma([\pi_P])$ is in $H'$. Now we compute $\psi(x,[\pi_P])=\psi_P(x,\pi_P).$ By construction $P$ does not split in $K'$ so $x$ is not in $\mathscr{P} K_P.$ As $p$ does not divide $v_P(a_P)=1$, Lemma \ref{local pairing} implies that $\psi(x,a_P)\neq0$. This is equivalent to saying $\sigma([\pi_P])$ does not fix $\mathscr{P}^{-1}x$ and hence $H=H'$. 
				\end{proof}

				\subsection{Valuation and Deg}
				Let $\tilde{k}_0=\mathbb{F}_p^{sep}(T)$. Then $G(\tilde{k}_0/k_0)$ is topologically generated by the map $x\mapsto x^p$, namely the Frobenius, which we shall denote by $\varphi$, and is isomorphic to $\hat{\mathbb{Z}}$ via the map $\varphi\mapsto 1$. Denote this isomorphism $deg$.
				
				Then as in the abstract case for each finite extension $K$ of $k_0$ we can define an isomorphism 
				$$deg_K= \frac{1}{f_K}deg:G(\tilde{K}/K)\rightarrow \hat{\mathbb{Z}}$$ where $f_K=[K\cap \tilde{k}_0:k_0]$.
				
				Now we define the map [ , $L/K$] for every abelian extension $L/K$ of global fields of characteristic $p$ by
				$$[\alpha,L/K]=\prod_P(\alpha_P,\text{ }L_P/K_P)$$
				where ( , $L_P/K_P$) is the local norm residue symbol. Note for each $\alpha=(\alpha_P)\in I_K$ , $(\alpha_P,$ $L_P/K_P)$ are equal to $1$ for almost all $P$ as almost all $\alpha_P$ are units and almost all extensions $L_P/K_P$ are unramified.
				\begin{Prop}\label{com}
					If $L/K$ and $L'/K'$ are two abelian extensions such that $K\subset K'$ and $L\subset L'$ then we have the commutative diagram		
					\begin{center}
						\begin{tikzpicture}
						\matrix(m) [matrix of math nodes, row sep=3em, column sep=4em, minimum width=2em]
						{
							I_{K'}&	G(L'/K')\\
							I_K&	G(L/K)\\
						};
						\path[-stealth]
						(m-1-1) edge node [left] {$N_{K'/K}$} (m-2-1)
						(m-1-1) edge node [above] { $[$ , $L'/K']$} (m-1-2)
						(m-2-1) edge node [above] { $[$ , $L/K]$} (m-2-2)
						(m-1-2) edge node [right] {} (m-2-2);
						\end{tikzpicture}
						\end {center}	
					\end{Prop}
					\begin{proof}
						Exactly same as the number field case.
					\end{proof}
					In the same way as characteristic $0$ case, we can define [ , $L/K$] for an infinite extension $L$ of $K$ by taking projective limit. It is clear that Proposition \ref{com} holds for infinite extension as well.
					
					We define homomorphism $$v_K:I_K\rightarrow \hat{\mathbb{Z}}$$ to be the composite of $[\text{ } ,\tilde{K}/K]$ and $deg_K$.
					\begin{Prop}\label{Reciprocity}
						For every principal element $a\in K$ we have
						$$[a,\tilde{K}/K]=1.$$
					\end{Prop}
					\begin{proof}
						Since [$N_{K/k_0}(a)$, $\tilde{k_0}/k_0]=[a,$ $\tilde{K}/K]|_{\tilde{k_0}}$, we can assume $K=\mathbb{F}_p(T)$. It suffices to prove $[a,\mathbb{F}_{p^n}(T)/\mathbb{F}_p(T)]=1.$ View $\mathbb{F}_{p^n}$ as $\mathbb{F}_p(\zeta_n)$ where $\zeta_n$ is a primitive $(p^n-1)$th root of unity. Then for each place $Q$ of $K$ and $a\in K$
						\[(a,K_Q(\zeta_n)/K_Q)\zeta_n=\zeta_n^{p^{v_Q(a)}}\] by Proposition \ref{main}. So $[a,K(\zeta_n)/K]\zeta_n=\prod_P((a,K_P(\zeta_n)/K_P)\zeta_n)=\prod_P\zeta_n^{p^{\mathrm{Deg}(P)v_Q(a)}}=\zeta_n$.
					\end{proof}
					Therefore $v_K:I_K\rightarrow \hat{\mathbb{Z}}$ induces a homomorphism
					$$v_K:C_K\rightarrow \hat{\mathbb{Z}}.$$
					
					\begin{cor}
						$v_K(C_K^0)=1.$ More precisely $ker(v_K)=C_K^0.$
					\end{cor}
					\begin{proof}
						First we prove that $N_{K/k_0}I_K^0\subset I_{k_0}^0$. Proposition \ref{norm} says for $\alpha\in I_K$, $N_{K/k_0}(\alpha)_P=\prod_{\mathfrak{P}|P}N_{K_\mathfrak{P}/(k_0)_P}(\alpha_\mathfrak{P})$. So using $\prod_{\mathfrak{P}}|\alpha_\mathfrak{P}|_\mathfrak{P}=1$ and $|\alpha_\mathfrak{P}|_\mathfrak{P}=|N_{K/k_0}\alpha_\mathfrak{P}|_P$ we obtain $N_{K/k_0}\alpha \in I_{k_0}^0$.
						
						So as in Proposition \ref{Reciprocity} we can assume $K=\mathbb{F}_p(T)$. Let $\alpha \in I_K^0$.
						At the end of the proof of Proposition \ref{Reciprocity}, we obtained that $(\alpha,K(\zeta_n)/K)=1$ if and only if $\sum_P \mathrm{Deg}(P)v_P(\alpha_P)=0$. This condition is equivalent to $\alpha \in C_k^0$.
						
						%Similarly to proposition \ref{Reciprocity} we can assume $k=k_0$. Let $(\alpha)=\prod \alpha_P \in I_k^0$. For all but finitely many places $P$, $\alpha_P$ are units so let $P_0,$ $P_1$, $\cdots,$ $P_m$ be list all places where the component of $\alpha$ aren't (necessarily) units where $P_0$ denote the infinite place. Let $a=\prod_{i=1}^{m} f_{P_i}^{v_i(\alpha_{P_i})}\in k \subset I_k$. Note $v_P(a)=v_P(\alpha_P)$ for all finite places $P$ by construction. As $\alpha \in I_k^0$ this also implies the valuation at infinite place also coincide. As [$\cdot:\hat{k}/k$] only depend on the valuation at each component $[\alpha:\hat{k}/k]=[a:\hat{k}/k]=1$.
					\end{proof}
					\begin{Prop}
						The map $v_K:C_K\rightarrow \hat{\mathbb{Z}}$ has image $\mathbb{Z}$ and $v$ is a henselian valuation with respect to deg.
					\end{Prop}
					
					\begin{proof}
						Recall that in Proposition \ref{Decomp} we obtained decomposition $C_K=C_K^0\times T_p^\mathbb{Z}$  where $T_P=[\pi_P]_P$ and $\pi_P$ is a prime element for some rational place $P$. Using Proposition \ref{main} we know that $[T_P,$ $\tilde{K}/K]=(\pi_P,\tilde{K}_P/K_P) =\varphi_{K_P}$. As $P$ is a rational place $\varphi_{K_P}= \varphi_{K}$. This automatically implies the first assertion as 
						$$v_K(C_K)=v_K(\Gamma_K)=v_K(T_p^\mathbb{Z})=deg_K(\varphi_K^\mathbb{Z})=\mathbb{Z}\subset\hat{\mathbb{Z}}.$$ 
						
						%We first show that im$(v_k)$ contains the Frobenius. Let $l/k$ be a finite subex tension of $\hat{k}/k$. Since ( ,$l_p/k_p):k_p^\times \rightarrow G(l_p/k_p)$ is surjective, $[I_k,l/k]$ contains all decomposition group $G(l_p/k_p)$. Therefore, all primes $p$ of $k$ are totally decomposed in the fixed field $M$ of $[I_k,l/k]$. Then $M=K$ and hence $[I_k,l/k]=G(l/k).$ From this it follows that $[I_k,\hat{k}/k]=[C_k,\hat{k}/k]$ is dense in $G(\hat{k}/k).$ On the other hand we have $C_k=C_k^0\times \Gamma_k.$ For this purpose we let $\Gamma_k$ belong to infinite place. It suffices to prove that $[\Gamma_k,\hat{k}/k]=\mathbb{Z}\subset\hat{\mathbb{Z}}.$ But this is evident from the calculation that has been done so far.
						Then the first condition for $v_K$ being the henselian valuation is satisfied since $v_K(C_K)=\mathbb{Z}$. By Proposition \ref{com}, the second condition is also satisfied as for every finite extension $L/K$
						$$v_K(N_{L/K}C_L)=deg_K[N_{L/K}I_L,\tilde{K}/K]=f_{L/K}deg_L[I_L,\tilde{L}/L]=f_{L/K}v_L(C_L)=f_{L/K}\mathbb{Z}.$$
					\end{proof}
					
					Recall that the Idele class groups $C_K$ satisfy the class field theory axiom and hence we get the following result.
					\begin{thm}
						For every Galois extension $L/K$ of function field over a finite field we have a canonical isomorphism $$r_{L/K}:G(L/K)^{ab}\cong C_K/N_{L/K}C_L.$$ Also the Existence Theorem holds.
					\end{thm}
					The inverse of $r_{L/K}$ yields the surjective homomorphism
					$$(\text{ },L/K):C_K\rightarrow G(L/K)^{ab}.$$ with kernel $N_{L/K}C_L$.
					
					\begin{Prop}
						If $L/K$ is an abelian extension and $P$ a prime of $K$, then the diagram
						\begin{center}
							\begin{tikzpicture}
							\matrix(m) [matrix of math nodes, row sep=3em, column sep=4em, minimum width=2em]
							{
								K_P^\times &	G(L_P/K_P)\\
								C_K &	G(L/K)\\
							};
							\path[-stealth]
							(m-1-1) edge node [left] {$[$ $]_P$} (m-2-1)
							(m-1-1) edge node [above] { $($ , $L_P/K_P)$} (m-1-2)
							(m-2-1) edge node [above] { $($ , $L/K)$} (m-2-2)
							(m-1-2) edge node [right] {} (m-2-2);
							\end{tikzpicture}
							\end {center}	
							is commutative. 
						\end{Prop}
						\begin{proof}
							Exactly same as number field case.
						\end{proof}
						
			As in the abstract case, we can take inverse limit of the global norm residue symbol to obtain map $$(\text{ }, K^{ab}/K):C_{K}\rightarrow G(K^{ab}/K).$$ The image of this map is not surjective unlike the number field case (for instance $(\text{ }$, $\tilde{K}/K)$ is not surjective since the image is isomorphic to $\mathbb{Z}$). Also the kernel of this map is clearly given as the intersection of all $N_{L/K}C_L$ which can be proved to be connected component of $1$ (\cite{AT}). This is non-trivial for number fields but not for global function fields. So $(\text{ }, K^{ab}/K):C_{K}\rightarrow G(K^{ab}/K)$ is injective.
			\newpage
\appendix
%\addcontentsline{toc}{chapter}{Appendix}
\addtocontents{toc}{\protect\contentsline{chapter}{Appendix:}{}}
	\section{Full proof of the Second inequality}
	
		In this appendix we fill the gap in the proof of the Second inequality shown previously. We aim to prove Proposition \ref{pairing}.
			
				For this proof, a certain result on derivation will be needed. We recall that a derivation $D$ of a field $L$ is an additive map of $L$ to itself which satisfies the Leibniz rule
				$$D(xy)=xD(y)+yD(x).$$ 
				\begin{Def}
					An element $x\in L$ is called a logarithmic derivative in $L$ if there exists an element $y \in L$ such that $x=\frac{D(y)}{y}$. 
				\end{Def}

			Let $D$ be a derivation on $L$. The additive group of all $x\in L$ for which $D^n(x)=0$ shall be denoted by $P_n$. Let $D^0$ be an identity which means $P_0=\{0\}.$  Also we denote the additive operator that multiplies the elements of $L$ by a given element $y$ of $L$ also by $y$. There may be confusion as to whether $y$ denotes an element or the operator. For this we shall write $Dy$ for the operator product $y$ followed by $D$ where as $D(y)$ means the effect $D$ has on $y$.
			
			\begin{lemma*}
				An element $x$ of $L$ is a logarithmic derivative of an element $y\in P_n\backslash P_{n-1}$ $(n>0)$ if and only if $(D+x)^n(1)=0$ but $(D+x)^{n-1}(1)\neq 0$.
			\end{lemma*}
			\begin{proof}
				We first start by noting that for a derivation $D$, $D(1)=0$ and $D(x^{-1})=-x^{-2}D(x)$. This follows from the definition of a derivation.
				
				Let $x=\frac{D(y)}{y}$ for some $y\in L$. For all $z\in L$ we have $(D+x)(z)=D(z) +\frac{D(y)}{y}z=y^{-1}D(y)=y^{-1}Dy(z).$ This means in operators that $D+x=y^{-1}Dy$ and consequently for any $n\geq0$ that 
				$$(D+x)^n=y^{-1}D^ny.$$
				Assume now in addition that $y\in P_n\backslash P_{n-1}$. Then $(D+x)^n(1)=y^{-1}D^ny(1)=y^{-1}D^n(y)=0$ and similarly $(D+x)^{n-1}\neq0$. Assume conversely that $(D+x)^n(1)=0$ and $(D+x)^{n-1}(1)\neq 0$. Put
				$$(D+x)^{n-1}(1)=y^{-1}.$$ Then $(D+x)(y^{-1})=0=-y^{-2}D(y)xy^{-1}$ and consequently $x=D(y)/y.$ Therefore $D+x=y^{-1}Dy,$ $(D+x)^{n-1}(1)=y^{-1}D^{n-1}(y)\neq 0$ and similarly $y^{-1}D^n(y)=0$. Hence $y\in P_n\backslash P_{n-1}.$
			\end{proof}
			Usually the derivation $D$ has the additional property that for each $x$ in $L$ we have $D^n(x)=0$ for some $n$. Then the above Lemma shows that an element $x\in L$ is a logarithmic derivative in $L$ if and only if $(D+x)^n(1)=0$ for some $n$. If $x$ is now an element of a subfield $K$ of $L$ that is stable under $D$, (i.e.$D(K)\subset K,$) and if $x$ is an element of $K$, then $x$ is a logarithmic derivative in $K$ if and only if it is a logarithmic derivative in $L$. This situation occurs in the following case: Let $F$ be a field of characteristic $p\geq 0$ and $L=F((t))$ the field of formal power series in $t$ with coefficients in $F$. If $D$ is the ordinary differentiation with respect to $t$ then $D^p=0$. Therefore we obtain:
			
				\begin{cor*}\label{derivation}
					Let $L=F((t))$ where $char(F)=p>0$ and $K$ be a subfield stable under the ordinary derivative $D$. Then an element $x\in K$ is a logarithmic derivative in $L$ if and only if it is a logarithmic derivative in $L$. 
				\end{cor*}
			
			Now we proceed to proof of Proposition \ref{pairing}.			
			\begin{Prop*}
				The kernel of $I_K$ in our pairing $\psi$ is precisely $K\cdot I_K^p$ and the kernel of $K$ is the additive group $\mathscr{P}K$.
			\end{Prop*}
			\begin{proof}
				We first observe that $K$ and $I_K^p$ lie trivially in the kernel and therefore so does $K\cdot I_K^p$. 
				
				Now suppose $\alpha \in I_K$ is in the kernel. Then the Adele $\zeta=\frac{1}{\alpha}\frac{d\alpha}{dt}$ has the property that $f(x\zeta)dt=0$ for all $x\in K$ and is therefore an element $y$ of $K$. Let $Q$ be a place of $K$ and let $t_Q$ be a prime. By taking $Q$ component of $a$ we obtain 
				$$y=\frac{1}{\alpha_Q}\frac{d\alpha_Q}{dt}.$$ $a$ is therefore a logarithmic derivative in $K_Q$. The field $K$ is a subfield of $K_Q$ which is stable under differentiation since $t\in K$. So by the Corollary above we obtain that $y$ is a logarithmic derivative in $K$. Write $y=\frac{1}{z}\frac{dz}{dt}$ where $z\in K^\times$. We obtain for our Idele: $\frac{1}{\alpha_Q}\frac{d\alpha_Q}{dt}=\frac{1}{z}\frac{dz}{dt}.$ Define another Idele $\beta =z^{-1}\alpha$. Just by explicit differentiation we obtain $\frac{d\beta}{dt}=0$. That is to say for each place $P$, $\frac{d\beta_P}{dt}=0$. This means each component is a $p$-th power. Consequently $\beta\in I_K^p$ and $\alpha \in K\cdot I_K^p$.
				
				Suppose $x=z^p-z$ for some $z\in K$. By the approximation theorem we can find an element $y\in K$ such that the Idele $y_P$ is as close to $1$ as we want at all places $P$ where $x$ has a pole. At these primes $\psi_P(x,y\alpha_P)=0$. Furthermore $\psi(x,\alpha)=\psi(x,y\alpha)$ since $y$ is in the kernel of $I_K$. If $Q$ is a place where $x$ has no pole then $x\in \mathcal{O}_P$ and by Lemma \ref{local pairing} $\psi_Q(x,y\alpha_P)=0$. We recall that $\psi$ is equal to sum of $\psi_P$ so $x$ is in the kernel of $K$. 
				
				Conversely let $x$ be in the kernel of $K$. By considering Ideles of the form $[a_P]$ where $a_P\in K_P$ for each place $P$, we obtain $\psi_P(x,\alpha_P)=0$. Let $Q$ be a place where $x$ has no pole. At such a prime, Lemma \ref{local pairing} applies and tells that $x=z_Q^p-z_Q$ with $z_Q\in K_Q$. Such a prime completely splits in the global extension $K(\mathscr{P}^{-1}x)$ and by Corollary \ref{fc1} $K(\mathscr{P}^{-1}x)=K$. That is to say $x\in \mathscr{P} K$.
			\end{proof}	 	 		 		
		
			 \newpage
			 \bibliographystyle{plain}
			 \bibliography{reference}
			\end{document}